\documentclass[12pt]{amsart}

\usepackage{amssymb}
\usepackage{graphicx}
\usepackage{color}
\usepackage{hyperref}
\usepackage[margin=1in]{geometry}
\usepackage{enumerate}
\usepackage{tikz}
\usetikzlibrary{arrows,automata,positioning}

\newtheorem{theorem}{Theorem}
\numberwithin{theorem}{section}

\theoremstyle{definition}

\theoremstyle{remark}

\theoremstyle{plain}
\newtheorem{lem}[theorem]{Lemma}
\newtheorem{thm}[theorem]{Theorem}
\newtheorem{prop}[theorem]{Proposition}
\newtheorem{cor}[theorem]{Corollary}

\theoremstyle{definition}
\newtheorem{defn}[theorem]{Definition}
\newtheorem{exmp}[theorem]{Example}

\theoremstyle{remark}
\newtheorem{rem}[theorem]{Remark}

\numberwithin{equation}{section}

\title{Thick points of log-correlated Gaussian fields do not depend on the mollifier}

\author{Kyle Ambrose}

\date{\today}

\begin{document}

\begin{abstract}
The log-correlated Gaussian field (LGF) on $\mathbb{R}^d$ ($d \geq 2$) is a centered Gaussian random tempered distribution, defined modulo additive constants, whose covariance kernel is $\log(1/|x-y|)$. In two dimensions, the LGF coincides with the whole-plane Gaussian Free Field (GFF). Because the field is a distribution, studying its pointwise behavior requires regularization via convolution with a mollifier. A point is called $\alpha$-thick if the mollified field at that point grows like $\alpha \log(1/\epsilon)$ as the mollification scale $\epsilon \to 0$. A natural question is whether the set of $\alpha$-thick points depends on the choice of mollifier. We prove that for any two admissible mollifiers $\rho$ and $\sigma$ satisfying some mild conditions, the thick point sets coincide almost surely. The result holds in all dimensions $d \geq 2$, and in the special case $d = 2$ extends to the zero-boundary GFF on any open domain with harmonically non-trivial boundary via the Markov property. 
\end{abstract}

\maketitle

\tableofcontents

\section{Introduction}
The Gaussian free field (GFF) is the natural $d$-dimensional analogue of Brownian motion. While Brownian motion is a canonical random function, the GFF is a canonical random distribution whose covariance kernel is the Green's function. In two dimensions, this gives the GFF a logarithmic covariance structure. The GFF arises in many contexts, including conformal field theory, Schramm-Loewner evolution, and Liouville quantum gravity. Much of the rich mathematics comes from the logarithmic correlation.

For dimensions $d \geq 3$, the standard GFF has polynomial correlations and lacks this logarithmic behavior. However, one can define the log-correlated Gaussian field (LGF), which retains the logarithmic covariance kernel $\log(1/|x-y|)$ in every dimension $d \geq 2$. In dimension two, the LGF coincides with the usual whole-plane GFF modulo additive constants. 

While the LGF is a random distribution and not a pointwise-defined function, one can study its local behavior by regularizing via convolution with a mollifier. For a mollifier $\rho$, the mollified field $h_\epsilon^\rho(z)$ is obtained by averaging the field against $\rho$, rescaled to a ball of radius $\epsilon$ centered at $z$. For log-correlated fields, the variance of this mollified average grows like $\log(1/\epsilon)$. A point $z$ is called an $\alpha$-thick point if $h_\epsilon^\rho(z)$ grows like $\alpha \log(1/\epsilon)$ as $\epsilon \to 0$.

A natural question is whether the set of thick points depends on the choice of mollifier. In this paper, we prove that for the LGF on $\mathbb{R}^d$ in any dimension $d \geq 2$, the set of $\alpha$-thick points is independent of the choice of mollifier, provided the mollifier satisfies some mild conditions. Our result also extends to the zero-boundary GFF on open domains with harmonically non-trivial boundary in two dimensions via the Markov property.

We now recall the precise definitions needed to state our main theorems. Further background and details appear in Section~\ref{sec:preliminaries}.

Let $d \geq 2$. Let $\mathcal S(\mathbb R^d)$ denote the Schwartz space on $\mathbb R^d$, and write
$$
\mathcal S_0(\mathbb R^d)=\Bigl\{f\in\mathcal S(\mathbb R^d):\int_{\mathbb R^d} f(x)dx=0\Bigr\}.
$$

\begin{defn}\label{def:LGF}
The \emph{log-correlated Gaussian field on $\mathbb{R}^d$} is the centered Gaussian random tempered distribution modulo additive constants with covariance
$$
\operatorname{Cov}\bigl[(h,f_1),(h,f_2)\bigr]
=
\int_{\mathbb R^d\times\mathbb R^d}
\log\frac1{|x-y|}
f_1(x)f_2(y)dxdy
$$
for all $f_1,f_2\in\mathcal S_0(\mathbb R^d)$.
\end{defn}

\begin{defn}\label{def:mollifier}
An \emph{admissible mollifier} is a nonnegative Radon measure $\rho$ on $\mathbb R^d$ such that
\begin{enumerate}[(i)]
    \item $\operatorname{supp}(\rho) \subseteq \overline{B(0,1)}$,
    \item $\rho(\mathbb{R}^d) = 1$,
    \item $\displaystyle\iint_{\mathbb{R}^d \times \mathbb{R}^d} \left|\log \frac{1}{|x-y|}\right| \rho(dx)\rho(dy) < \infty$.
\end{enumerate}
\end{defn}
\begin{rem}\label{rem:compactSupportConvenience}
We include the compact support condition for convenience and expect that it can be relaxed to a sufficiently rapid decay hypothesis at infinity, but we do not pursue this here.
\end{rem}

The finite-energy condition (iii) ensures that the distributional pairing below is well defined. For $\epsilon>0$ and $z\in\mathbb R^d$, let $\rho_{z,\epsilon}$ denote the pushforward of $\rho$ under $y\mapsto z+\epsilon y$, and define the mollified field by
$$
h_\epsilon^\rho(z)
=
(h,\rho_{z,\epsilon}).
$$
This pairing is made precise in Section \ref{sec:mollification}. After fixing the additive constant of $h$, for example by requiring its unit spherical average to vanish, the mollified field is a well-defined random variable. 

\begin{defn}\label{def:thickPoints}
For an admissible mollifier $\rho$ whose mollified field $h_\epsilon^\rho(z)$ has a modification jointly continuous in $(z,\epsilon)$, and for $\alpha\in \mathbb{R}$, define
$$
\begin{aligned}
\mathcal T_{\alpha,\liminf}^\rho
&=
\Bigl\{
z \in \mathbb{R}^d :
\liminf_{\epsilon \to 0}
\frac{h_\epsilon^\rho(z)}{\log(1/\epsilon)}
= \alpha
\Bigr\},
\\
\mathcal T_{\alpha,\limsup}^\rho
&=
\Bigl\{
z \in \mathbb{R}^d :
\limsup_{\epsilon \to 0}
\frac{h_\epsilon^\rho(z)}{\log(1/\epsilon)}
= \alpha
\Bigr\},
\\
\mathcal T_{\alpha,\lim}^\rho
&=
\Bigl\{
z \in \mathbb{R}^d :
\lim_{\epsilon \to 0}
\frac{h_\epsilon^\rho(z)}{\log(1/\epsilon)}
= \alpha
\Bigr\}.
\end{aligned}
$$
\end{defn}
\begin{rem}\label{rem:additiveConstant}
The thick point sets are independent of the choice of additive constant of $h$. This is because changing the additive constant by $C$ changes the ratio $h_\epsilon^\rho(z)/\log(1/\epsilon)$ by $C/\log(1/\epsilon)$, which tends to $0$ as $\epsilon \to 0$.
\end{rem}
Thick point sets have been studied extensively in the literature (see \cite{aru2022thickpointsplanargff}, \cite{ding2026percolationthickpointslogcorrelated}, \cite{HuMillerPeres2010}). In two dimensions, thick points are usually defined using the circle average $\sigma$ (the uniform measure on the unit circle). For this choice, the Hausdorff dimension of $\mathcal{T}_{\alpha,\lim}^\sigma$ is $2 - \alpha^2/2$, and the set is almost surely empty for $\alpha > 2$ \cite{HuMillerPeres2010}.

Our main result applies to mollifiers satisfying a second-moment regularity estimate. For $p\in(0,1]$, we say that an admissible mollifier $\rho$ satisfies
\emph{Condition \eqref{eq:secondMoment}} if
\begin{equation}\label{eq:secondMoment}
\mathbb E\bigl[\bigl(h_\epsilon^\rho(z)-h_r^\rho(w)\bigr)^2\bigr]
\leq C
\left(
\frac{|(z,\epsilon)-(w,r)|}{\epsilon\wedge r}
\right)^p
\tag{$A_p$}
\end{equation}
for all $z,w\in\mathbb R^d$ and all $\epsilon,r\in(0,1]$ with
$\frac{1}{2}\leq \epsilon/r\leq 2$, where $C$ may depend on $\rho$, $p$, and $d$.

Condition \eqref{eq:secondMoment} is important because it implies a H\"older-continuity-type estimate for the mollified field $h_\epsilon^\nu(z)$ (see Lemma \ref{lem:jointContinuity}). 

The following is our main theorem.

\begin{thm}\label{thm:main}
Let $d \geq 2$, and let $h$ be the log-correlated Gaussian field on $\mathbb{R}^d$. Let $\rho$ and $\sigma$ be admissible mollifiers in the sense of Definition \ref{def:mollifier}. Suppose that there exists $p\in(0,1]$ such that $\rho$ and $\sigma$ both satisfy Condition \eqref{eq:secondMoment}. Then almost surely,
$$
\mathcal T_{\alpha,\star}^\rho
=
\mathcal T_{\alpha,\star}^\sigma
\qquad
\text{for every }
\alpha\in\mathbb R
\text{ and every }
\star\in\{\liminf,\limsup,\lim\}.
$$
\end{thm}

In proving Theorem~\ref{thm:main}, we will actually prove a stronger statement.
\begin{prop}\label{prop:uniformDiffGoal}
Let $p\in(0,1]$, and let $\rho,\sigma$ be admissible mollifiers satisfying Condition \eqref{eq:secondMoment}. Define $X(z,\epsilon):=h_\epsilon^\rho(z)-h_\epsilon^\sigma(z)$. Then for every bounded open set $U\subset\mathbb R^d$,
$$
\sup_{z\in U}\frac{|X(z,\epsilon)|}{\log(1/\epsilon)}
\to0\qquad\text{almost surely as } \epsilon \to 0.
$$
\end{prop}
This uniform $o(\log(1/\epsilon))$ bound implies the equality of thick point sets asserted in Theorem~\ref{thm:main}, since it forces the normalized liminf, limsup, and, when it exists, limit of the two mollified fields to agree at every point. The following sufficient condition is useful for checking the hypothesis of Theorem~\ref{thm:main}.

\begin{thm}\label{thm:sufficient}
Let $d \geq 2$, let $p\in(0,1]$, and let $h$ be the log-correlated Gaussian field on $\mathbb{R}^d$. Let $\rho$ be an admissible mollifier. If there
exists $\epsilon_0>0$ such that
$$
\sup_{a\in B(0,1+\epsilon_0)}
\int_{\mathbb{R}^d} \frac{1}{|a-y|^p} \rho(dy) < \infty,
$$
then $\rho$ satisfies Condition \eqref{eq:secondMoment}.
\end{thm}
The integrability condition appearing in
Theorem~\ref{thm:sufficient} holds for many natural classes of mollifiers, including bounded compactly supported densities and spherical averages. In Section \ref{sec:zeroBoundary}, we prove the analogous statement to Theorem \ref{thm:main} for the zero-boundary GFF.

\begin{defn}\label{def:zeroBoundaryGFF}
Let $D\subset\mathbb R^2$ be an open domain with harmonically
non-trivial boundary, and let $G_D$ be its Dirichlet Green's function.
The \emph{zero-boundary Gaussian free field} $h^D$ on $D$ is the centered
Gaussian random distribution with covariance
$$
\operatorname{Cov}\bigl[(h^D,f),(h^D,g)\bigr]
=\iint_{D\times D}G_D(x,y) f(x)g(y)dxdy
$$
for all $f,g\in C_c^\infty(D)$.
\end{defn}

In two dimensions, $G_D(x,y)$ behaves like $\log(1/|x-y|)$ near the diagonal. For an admissible mollifier $\rho$, define $\mathcal T_{\alpha,\star}^{D,\rho}$ as in Definition~\ref{def:thickPoints}, with $h^D$ in place of $h$, and with $\epsilon$ small enough so that the averaging region lies inside $D$.

\begin{thm}\label{thm:main-zero}
Let $h$ be the log-correlated Gaussian field on $\mathbb{R}^2$, let $D\subset\mathbb R^2$ be an open domain with harmonically non-trivial boundary, and let $\rho,\sigma$ be admissible mollifiers (Definition \ref{def:mollifier}). If both $\rho$ and $\sigma$ satisfy Condition \eqref{eq:secondMoment} with respect to $h$, then almost surely $$
\mathcal T_{\alpha,\star}^{D,\rho}
=\mathcal T_{\alpha,\star}^{D,\sigma}
\qquad\text{for every }
\alpha\in\mathbb R \text{ and every }\star\in\{\liminf,\limsup,\lim\}.
$$
\end{thm}

Section~\ref{sec:preliminaries} gives background on the two-dimensional GFF, the log-correlated Gaussian field, and mollification. In Section~\ref{sec:mollifiers}, we prove Theorem~\ref{thm:sufficient} and give a simplified criterion for radial mollifiers. Section~\ref{sec:jointContinuity} contains a modified
Kolmogorov-Chentsov theorem and uses Condition~\eqref{eq:secondMoment} to obtain a H\"older-type continuity estimate for the mollified field. Section~\ref{sec:mainTheorem} proves Theorem~\ref{thm:main}. Finally, Section~\ref{sec:zeroBoundary} proves Theorem~\ref{thm:main-zero}.

\section{The Gaussian Free Field and Log-Correlated Fields}\label{sec:preliminaries}

In this section we recall the Gaussian free field (GFF) and the log-correlated Gaussian field on $\mathbb{R}^d$ (LGF), and the notion of mollification. More detailed treatments can be found in \cite{berestycki2025gaussianfreefieldliouville}, 
\cite{duplantier2014logcorrelatedgaussianfieldsoverview}, \cite{lodhia2016fractionalgaussianfieldssurvey}, \cite{WernerPowell2021}, and \cite{sheffield2006gaussianfreefieldsmathematicians}.

\subsection{The Gaussian Free Field in two dimensions}

Let $D \subset \mathbb{R}^2$ be an open domain with harmonically non-trivial boundary, and let $G_D$ be the Dirichlet Green's function on $D$. We use the zero-boundary Gaussian free field $h^D$ from Definition~\ref{def:zeroBoundaryGFF}.

In two dimensions, the Dirichlet Green's function $G_D$ has the asymptotic behavior
$$G_D(x,y) = \log \frac{1}{|x-y|} + O(1)$$ 
as $y \to x$, for any $x \in D$.
See \cite{berestycki2025gaussianfreefieldliouville} for a comprehensive treatment of the Green's function.

One can obtain the whole-plane Gaussian free field on $\mathbb{R}^2$ by taking a local limit of the zero-boundary GFF on large disks, modulo additive constants. This field is exactly the case $d=2$ of the log-correlated Gaussian field from Definition~\ref{def:LGF}. We denote the whole-plane GFF on $\mathbb{R}^2$ by $h$.

An important property of this field is the Markov property.

\begin{lem}[{\cite[Lemma 2.2]{GMS}, \cite[Proposition 2.8]{MillerSheffield2017}}]\label{lem:markov}
Let $h$ be the whole-plane GFF on $\mathbb{R}^2$ (the
log-correlated Gaussian field from Definition~\ref{def:LGF} with
$d=2$), with the additive constant chosen so that the average of $h$ over the unit circle is $0$.
For each open set $U\subset\mathbb{R}^2$ with harmonically non-trivial boundary, we have the decomposition
$$
h = h^U + \phi,
$$
where $\phi$ is a random distribution which is harmonic on $U$ and is determined by $h|_{\mathbb R^2\setminus U}$, and $h^U$ is independent of $\phi$ and has the law of a zero-boundary GFF on $U$ minus its average over $\partial \mathbb{D}\cap U$.
If $U$ is disjoint from $\partial \mathbb{D}$, then $h^U$ is a zero-boundary GFF and is independent of $h|_{\mathbb R^2\setminus U}$.
\end{lem}

\subsection{The Log-Correlated Gaussian Field}

The log-correlated Gaussian field was defined in
Definition~\ref{def:LGF}. It is the $d$-dimensional analogue of the two-dimensional whole-plane GFF with logarithmic correlations in every
dimension $d\ge 2$.

In dimension $d = 2$, the LGF coincides modulo additive constants with the whole-plane GFF. For $d \geq 3$, the LGF is a different object from the standard GFF, as the standard GFF in $d \geq 3$ has covariance kernel $|x-y|^{2-d}$, while the LGF retains the logarithmic correlation structure $\log(1/|x-y|)$ at all scales \cite{duplantier2014logcorrelatedgaussianfieldsoverview, lodhia2016fractionalgaussianfieldssurvey}. Since $C_c^\infty(\mathbb R^d)\cap\{\int f=0\}$ is dense in $\mathcal S_0(\mathbb R^d)$, the LGF is determined by its covariance on smooth compactly supported functions with zero mean.

The LGF on $\mathbb{R}^d$ satisfies scale and translation invariance.

\begin{lem}\label{lem:scalingInvariancePrelim}
Let $h$ be an LGF on $\mathbb{R}^d$. For every $r > 0$ and $a \in \mathbb{R}^d$,
$$
h(r \cdot + a) \overset{d}{=} h(\cdot),
$$
where equality is in the sense of distributions modulo additive constants.
\end{lem}

\begin{proof}
We interpret $h(r \cdot + a)$ as a distribution via
$$
(h(r \cdot + a), f) := r^{-d} (h, f(r^{-1}(\cdot - a))) \qquad \text{for } f \in \mathcal{S}_0(\mathbb{R}^d).
$$
It suffices to show that both fields have the same covariance. For any $f_1, f_2 \in \mathcal{S}_0(\mathbb{R}^d)$, making the change of variables $u = r^{-1}(x-a)$, $v = r^{-1}(y-a)$ gives
\begin{align*}
\operatorname{Cov}\bigl[(h(r \cdot + a), f_1), (h(r \cdot + a), f_2)\bigr]
&= r^{-2d} \iint \log \frac{1}{|x-y|} f_1(r^{-1}(x-a)) f_2(r^{-1}(y-a))   dx   dy \\
&= \iint \log \frac{1}{|r u - r v|} f_1(u) f_2(v)   du   dv \\
&= \iint \log \frac{1}{|u-v|} f_1(u) f_2(v)   du   dv \\
&= \operatorname{Cov}[(h,f_1),(h,f_2)],
\end{align*}
where the $\log r$ term vanishes because $\int f_i = 0$. Since both are centered Gaussian fields with the same covariance, the result follows.
\end{proof}

We fix the global additive constant of $h$ by requiring that the average of $h$ over the unit sphere is zero. As shown in \cite[Section 11.1]{lodhia2016fractionalgaussianfieldssurvey}, the spherical average of the LGF is well-defined. We will consider the field $\tilde h = h - h_1(0)$ where $h_1(0)$ is the unit spherical average of $h$. Because all quantities of interest in this paper are unchanged when a constant is added to the field, we will, by a slight abuse of notation, denote this normalized field again by $h$. This particular choice of normalization does not affect the set of thick points defined below, as it only shifts $h$ by a random constant. With this normalization fixed, the pairing $(h, \rho)$ is well-defined for any smooth, compactly supported function $\rho$ with $\int_{\mathbb{R}^d} \rho(x) dx = 1$.

\subsection{Mollification}\label{sec:mollification}

Since the LGF is a distribution, it is not defined pointwise. To study its local behavior, we regularize by averaging against mollifiers. While the field is initially defined on smooth test functions, we will need to consider mollifiers that are measures rather than smooth functions (for example, the circle average in two dimensions). The following approximation argument shows that the pairing $(h, \rho)$ can be extended to any measure $\rho$ satisfying a mild energy condition.

\begin{defn}\label{def:energyCondition}
A nonnegative Radon measure $\rho$ on $\mathbb{R}^d$ has \emph{finite logarithmic energy} if
\begin{equation}\label{eq:energyCondition}
\iint_{\mathbb{R}^d \times \mathbb{R}^d} \left|\log \frac{1}{|x-y|}\right| \rho(dx)\rho(dy) < \infty.
\end{equation}
\end{defn}

\begin{lem}\label{lem:measureApproximation}
Let $\rho$ be a nonnegative Radon measure satisfying the finite logarithmic energy condition \eqref{eq:energyCondition} and $\rho(\mathbb{R}^d) = 1$. Let $\psi$ be a smooth, nonnegative, symmetric function supported in $B(0,1)$ with $\int_{\mathbb{R}^d} \psi(x) dx = 1$. For $n \geq 1$, define
$$
\rho_n(y) = (\rho * \psi_n)(y) = \int_{\mathbb{R}^d} \psi_n(y - z) \rho(dz), \qquad \psi_n(z) = n^d \psi(nz).
$$
Then each $\rho_n$ is a smooth, compactly supported function with $\int_{\mathbb{R}^d} \rho_n(x) dx = \rho(\mathbb{R}^d)$. Furthermore, for the log-correlated Gaussian field $h$ on $\mathbb{R}^d$, the pairings $(h,\rho_n)$ are well defined and $\{(h,\rho_n)\}_{n\geq 1}$ is a Cauchy sequence in $L^2(\mathbb{P})$.
\end{lem}

See \cite[Section 3.2]{berestycki2025gaussianfreefieldliouville} for details of the proof. Using Lemma \ref{lem:measureApproximation}, we can extend the pairing $(h, \rho)$ to any measure satisfying the energy condition.

\begin{defn}\label{def:measureExtension}
Let $\rho$ be a nonnegative Radon measure satisfying the finite logarithmic energy condition \eqref{eq:energyCondition} and $\rho(\mathbb{R}^d)=1$. Let $\rho_n$ be the smooth approximations from Lemma \ref{lem:measureApproximation}. Define
$$
(h, \rho) := \lim_{n \to \infty} (h, \rho_n) \quad \text{in } L^2(\mathbb{P}),
$$
where the pairing $(h, \rho_n)$ on the right is well-defined because $\rho_n$ is smooth and $(h,\rho_n)$ is Cauchy. This definition is independent of the choice of $\psi$.
\end{defn}

We recall from Definition \ref{def:mollifier} the class of admissible mollifiers used throughout the paper. Condition (iii) is precisely the finite logarithmic energy condition of Definition \ref{def:energyCondition}, so the pairing $(h, \rho)$ is well-defined by Definition \ref{def:measureExtension}.

For $\epsilon > 0$, define the scaled measure $\rho_\epsilon(A) = \rho(A/\epsilon)$. For $x \in \mathbb{R}^d$, let $\rho_{x,\epsilon}$ be the translate of $\rho_\epsilon$ by $x$, so that
$$
\int f(y) \rho_{x,\epsilon}(dy) = \int f(x + \epsilon y) \rho(dy)
$$
for bounded measurable $f$. The mollified field is then defined by
\begin{equation}\label{eq:mollifiedField}
h_\epsilon^\rho(x) := (h, \rho_{x,\epsilon}).
\end{equation}
We will also write $\int h(y)\rho_{x,\epsilon}(dy) = h*\rho_\epsilon(x) = (h,\rho_{x,\epsilon})$ with the understanding that this is an abuse of notation justified by Definition \ref{def:measureExtension}. 

If $\rho$ satisfies Condition \eqref{eq:secondMoment} for some $p\in(0,1]$, then Lemma \ref{lem:jointContinuity} gives a modification of
the process $(z,\epsilon)\longmapsto h_\epsilon^\rho(z)$ which is jointly continuous in $(z,\epsilon)$ on compact subsets of $\mathbb R^d$. Throughout the paper, whenever Condition \eqref{eq:secondMoment} is assumed, we
replace $h_\epsilon^\rho(z)$ with such a continuous modification.

\section{Regularity of Mollifiers}\label{sec:mollifiers}

In this section we prove Theorem~\ref{thm:sufficient} and give a simplified criterion for radial mollifiers.

\subsection{General Mollifiers}

\begin{proof}[Proof of Theorem~\ref{thm:sufficient}]
It suffices to prove
$$
\mathbb E\bigl[\bigl(h_\epsilon^\rho(z)-h_r^\rho(w)\bigr)^2\bigr]
\leq C\left(\frac{|z-w|+|\epsilon-r|}{\epsilon\wedge r}\right)^p,
$$
because $|z-w|+|\epsilon-r|\leq \sqrt2 |(z,\epsilon)-(w,r)|$.
Assume $r\leq\epsilon$ and $\frac{r}{\epsilon}\geq \frac{1}{2}$. Define
$$
\lambda:=\frac{r}{\epsilon}\in\left[\frac{1}{2},1\right], \qquad 
\xi:=\frac{z-w}{r}.
$$
The mollified field is $h_\epsilon^\rho(z)=\int h(x)\rho_{z,\epsilon}(dx)$.
Since both $\rho_{z,\epsilon}$ and $\rho_{w,r}$ are probability measures, their difference $\rho_{z,\epsilon}-\rho_{w,r}$ has total mass zero, so
applying the covariance formula for the LGF on $\mathbb{R}^d$ gives
\begin{equation}\label{firstcov}
I:= \mathbb E\bigl[\bigl(h_\epsilon^\rho(z)-h_r^\rho(w)\bigr)^2\bigr]
= \iint \log\frac{1}{|x-y|}
(\rho_{z,\epsilon}-\rho_{w,r})(dx)
(\rho_{z,\epsilon}-\rho_{w,r})(dy).
\end{equation}

We perform a change of variables with $x=w+r u$ and $y=w+r v$. For $\rho_{w,r}$, we have by definition $\int f(x)\,\rho_{w,r}(dx) = \int f(w+r u)\rho(du)$. So under this change $\rho_{w,r}$ becomes simply $\rho$. 

For $\rho_{z,\epsilon}$, we use $z = w + r\xi$ and $\epsilon = r/\lambda$.
$$
\int f(x)\rho_{z,\epsilon}(dx) = \int f(z + \epsilon y)\rho(dy) = \int f\!\left(w + r\xi + \frac{r}{\lambda}\,y\right)\rho(dy)$$
$$
\int f(x)\rho_{z,\epsilon}(dx)= \int f \bigl(w + r(\xi + \lambda^{-1} y)\bigr)\rho(dy)= \int f(w + r u)\rho_{\xi,\lambda^{-1}}(du),
$$
So after the change of variables $x = w + r u$,
the measure $\rho_{z,\epsilon}$ turns into $\rho_{\xi,\lambda^{-1}}$. Since $|x-y|=r|u-v|$, \eqref{firstcov} becomes 
$$
I = 
\iint \log\frac{1}{|u-v|}(\rho_{\xi,\lambda^{-1}}-\rho)(du)(\rho_{\xi,\lambda^{-1}}-\rho)(dv)
$$
where the $-\log r$ disappears because $\iint -\log r(\rho_{\xi,\lambda^{-1}}-\rho)(du)(\rho_{\xi,\lambda^{-1}}-\rho)(dv) = 0$. Then, we have $I = I_1 + I_2 - 2 I_3$, where 
$$
\begin{aligned}
I_1&:=\iint \log\frac{1}{|u-v|}\rho_{\xi,\lambda^{-1}}(du)\rho_{\xi,\lambda^{-1}}(dv),\\
I_2&:=\iint \log\frac{1}{|u-v|} \rho(du)\rho(dv),\\
I_3&:=\iint \log\frac{1}{|u-v|} \rho_{\xi,\lambda^{-1}}(du)\rho(dv).
\end{aligned}
$$
We will estimate each of these integrals. Note that $I_2 < \infty$ is a constant dependent on $\rho$ by condition (iii) of Definition \ref{def:mollifier}. We can compute $I_1$ by the definition of $\rho_{\xi,\lambda^{-1}}$. 
$$
I_1
=\iint \log\frac{1}{|\xi+\lambda^{-1}s-(\xi+\lambda^{-1}t)|} \rho(ds)\rho(dt)
=\iint \log\frac{\lambda}{|s-t|} \rho(ds)\rho(dt)
=\log\lambda+I_2.
$$
Define
$$
U(x) := \int_{\mathbb{R}^d} \log\frac{1}{|x-y|}\rho(dy).
$$
We will estimate $I_3$.
$$
\begin{aligned}
I_3
&=\iint \log\frac{1}{|\xi+\lambda^{-1}s-v|} \rho(ds)\rho(dv)\\
&=\iint \left[\log\lambda+\log\frac1{|\lambda\xi+s-\lambda v|}\right]\rho(ds)\rho(dv)\\
&=\log\lambda+\int U(\lambda(v-\xi)) \rho(dv).
\end{aligned}
$$
Substituting into $I=I_1+I_2-2I_3$ gives
$$
I=(\log \lambda + I_2) + (I_2) - 2\left[\log \lambda + \int U(\lambda(v - \xi))\rho(dv)\right]$$
Now, using that $I_2 = \int U(v)\rho(dv)$
\begin{equation}\label{IntegralEstimatetwo}
I = -\log\lambda+2\int \bigl[U(v)-U(\lambda(v-\xi))\bigr] \rho(dv)
\end{equation}
We show that $U$ is $p$-H\"older continuous. Let 
$C_0 :=\sup_{a\in B(0,1+\epsilon_0)}\int_{\mathbb R^d}\frac{1}{|a-y|^p}\rho(dy)$. First, note that for any $a\in\mathbb R^d$, either $a \in \overline{B(0,1)}$, or $a \not \in \overline{B(0,1)}$. If $a \in \overline{B(0,1)}$, then by hypothesis, $\int _{\mathbb{R}^d}\frac{1}{|a-y|^p}\rho(dy) \leq C_0 <\infty$, and so $\rho(\{a\}) = 0$. If $a \not \in \overline{B(0,1)}$, then $a \not \in \operatorname{supp}\rho$, so $\rho(\{a\}) = 0$. In particular, this means that for any fixed $a,b \in \mathbb{R}^d$, the set $\{y \in \mathbb{R}^d: |a-y|=0 \text{ or } |b-y| = 0\}$ has $\rho$-measure zero. 

Next, we show that
$$
R(a):=\int_{\mathbb{R}^d}\frac{1}{|a-y|^p}\rho(dy)
$$
is globally bounded. If $a \in B(0,1+\epsilon_0)$, then $R(a)\leq C_0 < \infty$ by definition of $C_0$. If
$a\notin B(0,1+\epsilon_0)$, then $|a|>1+\epsilon_0$. By the reverse triangle inequality, for every
$y\in\operatorname{supp}\rho\subseteq\overline{B(0,1)},$ we have  $|a-y|\geq |a|-1>\epsilon_0$. Therefore,
$$
R(a) \leq \int_{\mathbb{R}^d}\frac{1}{\epsilon_0^p}\rho(dy)=\frac{1}{\epsilon_0^p}<\infty.
$$
Combining the two cases, taking the constant $C_1:=\sup_{a\in\mathbb{R}^d}R(a) \leq
\max\left(C_0,\frac{1}{\epsilon_0^p}\right)<\infty$, we have that
\begin{equation}\label{eq:Rbound}
R(a)=\int_{\mathbb{R}^d}\frac{1}{|a-y|^p}\rho(dy)
\leq C_1
\qquad\text{for all } a\in\mathbb{R}^d.
\end{equation}

Now take any $a,b\in\mathbb R^d$. For $\rho$-almost every $y$, both $|a-y|>0$ and $|b-y|>0$. Therefore, we can apply the estimate for $r_1,r_2 > 0$, 
$$|\log r_1-\log r_2|\leq C_2\frac{|r_1-r_2|^p}{\min(r_1,r_2)^p}$$
for some constant $C_2$ depending only on $p$. Using $r_1=|a-y|$ and $r_2=|b-y|$, we get
$$
|U(a)-U(b)|
\leq \int C_2\frac{\Bigl| |a-y| - |b-y|\Bigr|^p}{\min{(|a-y|,|b-y|)^p}}\rho(dy)$$
Using the reverse triangle inequality, \eqref{eq:Rbound}, and $\frac{1}{\min(|a-y|,|b-y|)^p} \leq \frac{1}{|a-y|^p} + \frac{1}{|b-y|^p}$, we get
$$|U(a)-U(b)|\leq C_2|a-b|^p \int \left( \frac{1}{|a-y|^p} + \frac{1}{|b-y|^p}\right)\rho(dy) \leq 2C_2C_1|a-b|^p 
$$
Next, we apply the $p$-H\"older continuity of $U$ to
\eqref{IntegralEstimatetwo}. For every
$v\in\operatorname{supp}\rho\subset \overline{B(0,1)}$, the
$p$-H\"older estimate with $C_3 := 2C_2C_1$ gives
\begin{equation}\label{eq:holderDifference}
|U(v)-U(\lambda(v-\xi))|
\leq C_3|v-\lambda(v-\xi)|^p
=C_3|(1-\lambda)v+\lambda\xi|^p.
\end{equation}
Since $|v|, \lambda\leq 1$,
\begin{equation}\label{eq:holderBasicBound}
|(1-\lambda)v+\lambda\xi|^p
\leq \left((1-\lambda)|v|+\lambda|\xi|\right)^p
\leq \left(1-\lambda+|\xi|\right)^p.
\end{equation}
Observe also that
\begin{equation}\label{eq:linearComparison}
1-\lambda+|\xi|\leq |\xi|+\frac{1}{\lambda}-1,
\end{equation}
because $1-\lambda \leq 1/\lambda-1$ for $0<\lambda\leq 1$.
Finally, since $\lambda \in [\frac{1}{2},1]$,
\begin{equation}\label{eq:loglambdaBound}
-\log\lambda \leq C_4(1-\lambda)^p
\leq C_4\left(\frac{1}{\lambda}-1\right)^p
\end{equation}
for some constant $C_4$. Inserting \eqref{eq:holderDifference},
\eqref{eq:holderBasicBound}, \eqref{eq:linearComparison},
and \eqref{eq:loglambdaBound} into
\eqref{IntegralEstimatetwo} gives
\begin{equation}\label{integralestimatethree}
\begin{aligned} 
|I|
&\leq (-\log\lambda)+2\int |U(v)-U(\lambda(v-\xi))|\rho(dv)\\
&\leq C_4\left(\frac{1}\lambda-1\right)^p+2\int C_3\left(|\xi|+\frac{1}\lambda-1\right)^p\rho(dv)\\
&\leq C_4\left(|\xi|+\frac{1}\lambda-1\right)^p+2C_3\left(|\xi|+\frac1\lambda-1\right)^p\\
&= C\left(|\xi|+\frac{1}\lambda-1\right)^p,
\end{aligned}
\end{equation}
where $C:=C_4+2C_3$. Recall the definitions $\xi=(z-w)/r$ and $1/\lambda=\epsilon/r$. 
$$
|\xi|+\frac1\lambda-1
=\frac{|z-w|}{r}+\frac{\epsilon-r}{r}
=\frac{|z-w|+|\epsilon-r|}{r}.
$$
Since $r=\epsilon\wedge r$, this is exactly the desired bound, which completes the proof.

\end{proof}

\begin{exmp}\label{exmp:bump}
If $\rho(dx)=f(x)dx$ where $f$ is a bounded, nonnegative function supported in $\overline{B(0,1)}$ with $\int f(x)dx=1$, then $\int \frac{1}{|a-y|^p}f(y)dy$ is uniformly bounded for all $p\in(0,1]$. Therefore, Theorem \ref{thm:sufficient} applies, and $\rho$ satisfies Condition~\eqref{eq:secondMoment} for all $p\in(0,1]$.
\end{exmp}

\subsection{Radial Mollifiers}
\begin{defn}\label{def:radial}
An admissible mollifier $\rho$ on $\mathbb{R}^d$ is called \emph{radial} if it is rotationally invariant.  Equivalently, there exists a probability measure $\nu$ on $[0,1]$ such that
$$
\rho = \int_0^1 \sigma_t  \nu(dt),
$$
where $\sigma_t$ is the normalized uniform measure on the sphere of radius $t$.  
\end{defn}

For radial mollifiers, checking the hypothesis of
Theorem~\ref{thm:sufficient} can be simplified. 

\begin{prop}\label{prop:radialSufficient}
Let $p\in(0,1]$, and let $\rho$ be a radial admissible mollifier on
$\mathbb{R}^d$. If
$$
\int_{\mathbb{R}^d}\frac{1}{|x|^p}\rho(dx)<\infty,
$$
then $\rho$ satisfies the hypothesis of Theorem \ref{thm:sufficient} with exponent $q$ for every $q\in(0,p]$, except in the case $d=2$ and $p=1$, where it holds for every $q\in(0,1)$. 
\end{prop}

\begin{proof}
Let $q$ be an exponent allowed by the statement. In particular, $q\le p$, and $d-2-q>-1$. Let $\sigma_t$ be the normalized uniform measure on the sphere of radius $t$, and let $\sigma:=\sigma_1$ denote the normalized uniform measure on the unit sphere $S^{d-1}\subset\mathbb R^d$. We first claim that there exists $C_{d,q}<\infty$ such that 
\begin{equation}\label{eq:sphereUniformBound}
\int_{S^{d-1}}\frac{1}{|x-u|^q} \sigma(du) \leq C_{d,q}
\qquad\text{for all }x\in\mathbb{R}^d.
\end{equation}

For $x\in\mathbb{R}^d$,
$$\int_{S^{d-1}}\frac{1}{|x-u|^q}\sigma(du)
=
\int_0^\infty q r^{-q-1}\sigma\bigl(S^{d-1}\cap B(x,r)\bigr)dr.$$
For $r\geq 1$, the trivial bound $\sigma\bigl(S^{d-1}\cap B(x,r)\bigr)\leq
\sigma(S^{d-1})=1 $ gives
$$
\int_1^\infty q r^{-q-1}\sigma\bigl(S^{d-1}\cap B(x,r)\bigr) dr\leq \int_1^\infty q r^{-q-1}dr = 1.
$$
There is a constant $A\ge 1$ depending only on $d$ such that, for every $u\in S^{d-1}$ and every
$0\leq t<\operatorname{diam}(S^{d-1})=2$,
\begin{equation}\label{ahlfors}
A^{-1}t^{d-1} \leq \sigma\bigl(S^{d-1}\cap B(u,t)\bigr) \leq A t^{d-1}.
\end{equation}
For $0<r<1$, if $S^{d-1}\cap B(x,r)$ is nonempty, choose
$$
u\in S^{d-1}\cap B(x,r).
$$
Then $S^{d-1}\cap B(x,r) \subseteq S^{d-1}\cap B(u,2r)$, so by \eqref{ahlfors},
$$
\sigma\bigl(S^{d-1}\cap B(x,r)\bigr)\leq \sigma\bigl(S^{d-1}\cap B(u,2r)\bigr) \leq A(2r)^{d-1}=A' r^{d-1}.
$$
This estimate is uniform in $x\in\mathbb{R}^d$. Therefore,
$$
\int_0^1 q r^{-q-1} \sigma\bigl(S^{d-1}\cap B(x,r)\bigr) dr
\leq A' q \int_0^1 r^{d-2-q} dr < \infty,
$$
because $d-2-q>-1$. Thus, we have proven \eqref{eq:sphereUniformBound}.

Write $\rho=\int_0^1\sigma_t\nu(dt)$ as in Definition~\ref{def:radial}. Now take $a\in \mathbb{R}^d$ and $t\in(0,1]$. By scaling and \eqref{eq:sphereUniformBound} applied with $x=a/t$, we get
$$
\int_{S^{d-1}}\frac{1}{|a-tu|^q}
\sigma(du)=t^{-q}\int_{S^{d-1}} \frac{1}{|a/t-u|^q}\sigma(du)\le C_{d,q}t^{-q}.
$$
Consequently,
$$
\int_{\mathbb{R}^d}\frac{1}{|a-y|^q}\rho(dy)
=\int_0^1\left(\int_{S^{d-1}}\frac{1}{|a-tu|^q}\sigma(du)\right)\nu(dt) \le
C_{d,q}\int_0^1 t^{-q}\nu(dt).
$$
Since $0<t\le 1$ and $q\leq p$, we have $t^{-q}\le t^{-p}$ and 
$$
\int_0^1t^{-q}\nu(dt) \leq \int_0^1t^{-p}\nu(dt)
=\int_{\mathbb{R}^d}\frac{1}{|x|^p}\rho(dx)
<\infty.
$$
This gives
$$
\sup_{a\in \mathbb{R}^d}\int_{\mathbb{R}^d}\frac{1}{|a-y|^q}\rho(dy)<\infty,
$$
which implies that $\rho$ satisfies the hypothesis of Theorem \ref{thm:sufficient} with any $\epsilon_0 >0$ and the exponent $q$.
\end{proof}

As an application, the spherical average mollifier satisfies Condition~\eqref{eq:secondMoment} for some $p\in(0,1]$. Let $\sigma$ denote the normalized surface measure on the unit sphere $S^{d-1}\subset\mathbb{R}^d$, for $d \geq 2$.

\begin{cor}\label{cor:sphereAverage}
For every $d\ge 2$, the spherical average $\sigma$ satisfies Condition~\eqref{eq:secondMoment} for some $p\in(0,1]$.
\end{cor}

\begin{proof}
The spherical average is the radial mollifier corresponding to $\nu=\delta_1$. 
$$
\int_{\mathbb{R}^d}\frac{1}{|x|}\sigma(dx)=\int_{S^{d-1}}1\,\sigma(dx)=1<\infty.
$$
Applying Proposition~\ref{prop:radialSufficient} with $p=1$, we obtain that for every $q\in(0,1)$, $\sigma$ satisfies the hypothesis of Theorem~\ref{thm:sufficient} with exponent $q$, and so $\sigma$ satisfies Condition~($A_q$) for every $q\in(0,1)$.
\end{proof}

The following is a useful fact that applies to the mollifier classes discussed above.

\begin{prop}\label{prop:linearCombination}
Let $p \in (0,1]$, and let $\rho_1,\dots,\rho_n$ be admissible mollifiers, each satisfying Condition \eqref{eq:secondMoment}. Let $c_1,\dots,c_n\geq 0$ with $\sum_{i=1}^n c_i = 1$, and define $\rho := \sum_{i=1}^n c_i \rho_i$. 
Then $\rho$ is an admissible mollifier and satisfies Condition \eqref{eq:secondMoment}.
\end{prop}

\begin{proof}
The admissibility conditions (i), (ii), and (iii) of Definition \ref{def:mollifier} are immediate. Note that $h_\epsilon^\rho(z) = \sum_{i=1}^n c_i h_\epsilon^{\rho_i}(z)$,
so by convexity of $t \mapsto t^2$,
$$
\mathbb{E}\bigl[\bigl(h_\epsilon^\rho(z) - h_r^\rho(w)\bigr)^2\bigr]
\leq \sum_{i=1}^n c_i \mathbb{E}\bigl[\bigl(h_\epsilon^{\rho_i}(z) - h_r^{\rho_i}(w)\bigr)^2\bigr]
\leq \left( \sum_{i=1}^n c_i C_i \right) \left(\frac{|(z,\epsilon) - (w,r)|}{\epsilon \wedge r}\right)^p.
$$
Therefore $\rho$ satisfies Condition \eqref{eq:secondMoment}.
\end{proof}

\section{Joint Continuity of the Mollified Field}\label{sec:jointContinuity}

In this section we prove a general joint continuity lemma for Gaussian processes satisfying Condition \eqref{eq:secondMoment}. The key technical tool is a modified Kolmogorov-Chentsov theorem, which we state first. It appears as Lemma C.1 in \cite{HuMillerPeres2010}.

\begin{lem}\label{lem:modifiedKolmogorov}
Let $U \subset \mathbb{R}^d$ be a bounded open set and let $X : U \times (0,1] \to \mathbb{R}$ be a random field. Suppose there exists $\alpha,\beta >0$ such that
$$
\mathbb{E}\bigl[\bigl|X(z,\epsilon) - X(w,r)\bigr|^\alpha\bigr] \leq C \left( \frac{|(z,\epsilon) - (w,r)|}{\epsilon \wedge r} \right)^{d+1+\beta}
$$
for all $z,w\in U$ and all $\epsilon,r\in(0,1]$ with
$\frac{1}{2}\leq \epsilon/r\le 2$. Then, for each $\zeta > \alpha^{-1}$ and $\gamma \in (0, \beta/\alpha)$, there exists a random variable $M>0$, finite almost surely, such that $X$ has a modification $\widetilde{X}$ satisfying
$$
|\widetilde{X}(z,\epsilon) - \widetilde{X}(w,r)|
\leq M \left( \log \frac{1}{\epsilon} \right)^\zeta
\frac{|(z,\epsilon) - (w,r)|^\gamma}{\epsilon^{(d+\beta)/\alpha}},
$$
for all $z,w \in U$ and $\epsilon,r \in (0,1]$ with $1/2 \leq \epsilon/r \leq 2$. 
\end{lem}
We will apply Lemma \ref{lem:modifiedKolmogorov} to the mollified fields satisfying Condition \eqref{eq:secondMoment}. 

\begin{lem}\label{lem:jointContinuity}
Let $p \in (0,1]$, let $U \subset \mathbb{R}^d$ be a bounded open set, and suppose $X(z,\epsilon)$ is a centered Gaussian process on $U \times (0,1]$ satisfying the second-moment bound
$$
\mathbb{E}\bigl[\bigl(X(z,\epsilon) - X(w,r)\bigr)^2\bigr] \leq C \left(\frac{|(z,\epsilon) - (w,r)|}{\epsilon \wedge r}\right)^p
$$
for all $z,w\in U$ and all $\epsilon,r\in(0,1]$ with
$\frac{1}{2}\le \epsilon/r\leq 2$. Then $X$ has a modification $\widetilde{X}$ such that for every $0 < \gamma < p/2$ and every $\zeta> 0$, there exists a random variable $M = M(p,\gamma, \zeta)$, finite almost surely, with the property that
\begin{equation}\label{eq:jointContinuityBound}
|\widetilde{X}(z,\epsilon) - \widetilde{X}(w,r)|
\leq M \left( \log \frac{1}{\epsilon} \right)^\zeta
\frac{|(z,\epsilon) - (w,r)|^\gamma}{\epsilon^{\frac{p}{2}}}
\end{equation}
for all $z,w \in U$ and $\epsilon,r \in (0,1]$ satisfying
$\frac{1}{2} \leq \frac{\epsilon}{r} \leq 2$.
In particular, $\widetilde{X}$ is jointly continuous in $(z,\epsilon)$ on $U \times (0,1]$.
\end{lem}

\begin{proof}
We use a similar method to that of Proposition 2.1 in \cite{HuMillerPeres2010}. In particular, it suffices to show that $X$ satisfies the hypothesis of Lemma \ref{lem:modifiedKolmogorov} for $\alpha,\beta$ arbitrarily large with $\beta/\alpha$ arbitrarily close to $p/2$. 

Let $\gamma\in(0,p/2)$ and $\zeta>0$ be given. Choose $\alpha$ large enough so that $\zeta > \frac{1}{\alpha}$, $\gamma < \frac{p}{2}- \frac{d+1}{\alpha}$. Let $\beta = \frac{p\alpha}{2}- (d+1)$. Then because the process is Gaussian, we have
$$
\mathbb{E}\bigl[ \bigl|X(z,\epsilon)-X(w,r)\bigr|^{\alpha} \bigr]\leq C_\alpha \Bigl(\frac{|(z,\epsilon)-(w,r)|}{\epsilon\wedge r}\Bigr)^{p\alpha/2}
= C_\alpha \Bigl(\frac{|(z, \epsilon)-(w,r)|}{ \epsilon\wedge r}\Bigr)^{d+1+\beta}.
$$
Therefore, the hypothesis of Lemma \ref{lem:modifiedKolmogorov} is satisfied. This yields a modification $\widetilde{X}$ and a random variable $M>0$, finite almost surely, such that
$$
|\widetilde{X}(z,\epsilon)-\widetilde{X}(w,r)|
\leq M\Bigl(\log\frac{1}{\epsilon}\Bigr)^\zeta
\frac{|(z,\epsilon)-(w,r)|^{\gamma}}{\epsilon^{\frac{d+\beta}{\alpha}}}$$
for all $z,w\in U$ and $\epsilon,r\in(0,1]$ with $\frac{1}{2}\leq \epsilon/r\leq 2$. Note that $\frac{d+\beta}{\alpha} = \frac{p}{2} - \frac{1}{\alpha} < \frac{p}{2}$. This proves \eqref{eq:jointContinuityBound}.
\end{proof}

\section{The Main Theorem for the LGF on $\mathbb{R}^d$}\label{sec:mainTheorem}

In this section we prove the main result of the paper. The thick point set of the log-correlated Gaussian field on $\mathbb{R}^d$ is independent of the choice of admissible mollifier. Throughout this section, we fix two admissible mollifiers $\rho$ and $\sigma$ and assume they both satisfy Condition \eqref{eq:secondMoment} for some fixed $p \in (0,1]$. 

\subsection{Preliminary Results}
We begin by defining the difference field which will be used to show the equivalence of the thick point sets.

\begin{defn}\label{def:Xepsilon}
For $\epsilon > 0$ and $z \in \mathbb{R}^d$, define
$$
X(z,\epsilon) := h_\epsilon^\rho(z) - h_\epsilon^\sigma(z),
$$
where $h_\epsilon^\rho$ and $h_\epsilon^\sigma$ are defined as in \eqref{eq:mollifiedField}.
\end{defn}

\begin{rem}
Throughout this section, we work with the LGF $h$ on $\mathbb{R}^d$ as a random distribution with global additive constant chosen so that the average of $h$ over the unit sphere is $0$. However, all quantities we consider are invariant of the chosen additive constant. The thick point sets are unchanged because the constant is divided by $\log(1/\epsilon)$, which tends to $0$, and the difference field $X(z,\epsilon) = h_\epsilon^\rho(z) - h_\epsilon^\sigma(z)$ cancels the constant because both mollifiers have total mass $1$. 
\end{rem}

We now state a technical result that will allow us to prove uniform convergence of $X(z,\epsilon)$.

\begin{lem}\label{lem:Xscaling}
For every $z \in \mathbb{R}^d$ and $\epsilon > 0$,
$$
\{X(z+\epsilon x,\epsilon)\}_{x \in B(0,1)}
\stackrel{d}{=}
\{X(x,1)\}_{x \in B(0,1)}.
$$
\end{lem}
\begin{proof}
Let $z\in\mathbb{R}^d$ and $\epsilon>0$. For $x\in B(0,1)$,
$$
X(z+\epsilon x,\epsilon) 
= \int_{\mathbb{R}^d} h\bigl(z+\epsilon(x+y)\bigr) (\rho-\sigma)(dy).
$$
By Lemma \ref{lem:scalingInvariancePrelim} the field $\{h(z+\epsilon x)\}_{x \in B(0,1)}$ has the same law as $\{h(x)\}_{x \in B(0,1)}$ modulo additive constants. But, note that the value of $X(z,\epsilon)$ is independent of the choice of additive constant representing $h$. Therefore, this implies
$$
\left\{\int_{\mathbb{R}^d} h\bigl(z+\epsilon(x+y)\bigr) (\rho-\sigma)(dy)\right\}_{x \in B(0,1)}
\stackrel{d}{=}
\left\{\int_{\mathbb{R}^d} h(x+y) (\rho-\sigma)(dy)\right\}_{x\in B(0,1)}
$$
The left-hand side is $\{X(z+\epsilon x,\epsilon)\}_{x\in B(0,1)}$, and the right-hand side is $\{X(x,1)\}_{x\in B(0,1)}$,
so this proves the lemma.
\end{proof}
We will use the following standard inequality for centered Gaussian processes, called the Borell-TIS inequality.

\begin{lem}[{\cite[Theorem 2.1.1]{AdlerTaylor}}]\label{lem:BorellTIS}
Let $\{Y_t\}_{t \in T}$ be a centered Gaussian Process, almost surely bounded on T. Then, 
$$
\mathbb{E}\Bigl[\sup_{t\in T}Y_t\Bigr]<\infty,
\qquad
\sigma_T^2:=\sup_{t\in T}\mathbb{E}\bigl[Y_t^2\bigr]<\infty,
$$
and for all $u > 0$, 
$$\mathbb{P}(\sup_{t\in T} |Y_t|> \mathbb{E}[\sup_{t\in T}Y_t] + u) \leq 2\exp(-\frac{u^2}{2\sigma_T^2}).$$
\end{lem}

\begin{proof}[Proof of Proposition~\ref{prop:uniformDiffGoal}]
Fix a bounded open set $U\subset\mathbb R^d$. Enlarging $U$ if necessary, we may assume that $\overline{B(0,1)}\subset U$. By Condition \eqref{eq:secondMoment} and Lemma \ref{lem:jointContinuity}, the mollified fields
$h^\rho_\epsilon(z)$ and $h^\sigma_\epsilon(z)$ admit jointly continuous modifications. Throughout the rest of this proof, we abuse notation and denote these modifications again by $h^\rho_\epsilon(z)$ and $h^\sigma_\epsilon(z)$. We define
$$
X(z,\epsilon):=h^\rho_\epsilon(z)-h^\sigma_\epsilon(z),
$$
so that $X$ is the corresponding jointly continuous modification of the difference field from Definition \ref{def:Xepsilon}. By Lemma \ref{lem:jointContinuity}, for every $\zeta>0$ and every $0<\gamma<p/2$, there exist almost surely finite random variables
$M_\rho$ and $M_\sigma$ such that
$$
\bigl|h^\nu_\epsilon(z)-h^\nu_r(w)\bigr|
\leq M_\nu\left(\log\frac1\epsilon\right)^\zeta
\frac{|(z,\epsilon)-(w,r)|^\gamma}{\epsilon^{p/2}}
$$
for $\nu=\rho,\sigma$, all $z,w\in U$, and all
$\epsilon,r\in(0,1]$ with $\frac{1}{2}\leq \epsilon/r\leq 2$. By the triangle inequality, $X$ satisfies the same estimate with $M:=M_\rho+M_\sigma$. In particular, $X$ is jointly continuous on
$U\times(0,1]$, so $\sup_{x\in B(0,1)}|X(x,1)|<\infty$
almost surely. This allows us to apply Lemma \ref{lem:BorellTIS}. Set
$$
m=\mathbb E\Bigl[\sup_{x\in B(0,1)}X(x,1)\Bigr]
\qquad\text{and}\qquad
\sigma^2=\sup_{x\in B(0,1)}
\mathbb E\bigl[|X(x,1)|^2\bigr].
$$
Then $m<\infty$ and $\sigma^2<\infty$, and for every $u>0$,
$$
\mathbb P\Bigl(\sup_{x\in B(0,1)}|X(x,1)|>m+u
\Bigr)\leq2\exp\Bigl(-\frac{u^2}{2\sigma^2}\Bigr).
$$
By Lemma \ref{lem:Xscaling}, for every $z\in\mathbb R^d$,
$\epsilon>0$, and $u>0$,
$$\mathbb P\Bigl(\sup_{x\in B(z,\epsilon)}|X(x,\epsilon)|>m+u\Bigr)
\leq2\exp\Bigl(-\frac{u^2}{2\sigma^2}\Bigr).
$$
Cover $U$ by balls $B(z_j,\epsilon)$, $j=1,\dots,N_\epsilon$, with $N_\epsilon\leq C_U\epsilon^{-d}$. Then
$$
\mathbb P\Bigl( \sup_{x\in U}|X(x,\epsilon)|>m+u \Bigr)
\leq 2C_U\epsilon^{-d} \exp\Bigl(-\frac{u^2}{2\sigma^2}\Bigr).
$$
Fix $\delta>0$. Let $u=\frac{\delta}{2}\log(1/\epsilon)$.
For sufficiently small $\epsilon$, $m+u<\delta\log(1/\epsilon)$. This gives
$$
\mathbb P\Bigl(\sup_{x\in U}|X(x,\epsilon)|>\delta\log(1/\epsilon)
\Bigr)\leq2C_U\epsilon^{-d}\exp\Bigl(-\frac{\delta^2(\log(1/\epsilon))^2}{8\sigma^2}\Bigr).
$$
Now take $\epsilon_n=1/n$. Then
$$
\mathbb P\Bigl(\sup_{x\in U}|X(x,\epsilon_n)|>\delta\log(1/\epsilon_n)\Bigr)\leq
2C_U n^d\exp\Bigl(-\frac{\delta^2(\log n)^2}{8\sigma^2}\Bigr).
$$
The right-hand side is summable, so Borel-Cantelli gives that almost surely, for all sufficiently large $n$,
$$
\frac{\sup_{x\in U}|X(x,\epsilon_n)|}{\log(1/\epsilon_n)}\leq\delta.
$$
Since $\delta>0$ was arbitrary,
$$
\frac{\sup_{x\in U}|X(x,\epsilon_n)|}{\log(1/\epsilon_n)}\longrightarrow 0\qquad\text{almost surely as }n\to\infty.
$$

It remains to upgrade this from the sequence
$\epsilon_n=1/n$ to arbitrary $\epsilon\to0$.
Let $\epsilon\in[\epsilon_{n+1},\epsilon_n]$.
Then $|\epsilon-\epsilon_n| \leq \epsilon_n-\epsilon_{n+1}\leq n^{-2}$ and $\frac{1}{2}
\leq \frac{\epsilon}{\epsilon_n} \leq 1$. Using the estimate for $X$ obtained above with $\zeta=1$ and $\gamma\in(\frac p4,\frac p2)$, there is an almost surely finite random variable $M$ such that
$$
\bigl|X(w,\epsilon)-X(w,\epsilon_n)\bigr| \leq M\left(\log\frac1{\epsilon_n}\right) \frac{|\epsilon-\epsilon_n|^\gamma}{\epsilon_n^{p/2}}
$$
for every $w\in U$. Using $|\epsilon-\epsilon_n|\leq n^{-2}$ and $\epsilon_n^{-p/2}=n^{p/2}$, we obtain
$$
\sup_{\epsilon\in[\epsilon_{n+1},\epsilon_n]}
\sup_{w\in U} \bigl|X(w,\epsilon)-X(w,\epsilon_n)\bigr|
\leq M(\log n)n^{\frac p2-2\gamma}.
$$
Since $\frac {p}{2}-2\gamma<0$,
$$
(\log n)n^{\frac p2-2\gamma}=o(\log n)
=o\left(\log\frac1{\epsilon_n}\right).
$$
Therefore,
$$
\sup_{\epsilon\in[\epsilon_{n+1},\epsilon_n]}\frac{\sup_{w\in U}\bigl|X(w,\epsilon)-X(w,\epsilon_n)\bigr|
}{\log(1/\epsilon_n)}\longrightarrow 0\qquad\text{as }n\to\infty.
$$
For $\epsilon\in[\epsilon_{n+1},\epsilon_n]$, the triangle inequality and taking the supremum over $\epsilon\in[\epsilon_{n+1},\epsilon_n]$ gives
$$
\sup_{\epsilon\in[\epsilon_{n+1},\epsilon_n]}
\sup_{w\in U}|X(w,\epsilon)|
\le
\sup_{\epsilon\in[\epsilon_{n+1},\epsilon_n]}
\sup_{w\in U}
\bigl|X(w,\epsilon)-X(w,\epsilon_n)\bigr|
+
\sup_{w\in U}|X(w,\epsilon_n)|.
$$
Since $\log(1/\epsilon)\ge \log(1/\epsilon_n)$ for
$\epsilon \in [\epsilon_{n+1},\epsilon_n]$, dividing by
$\log(1/\epsilon)$ and then replacing the denominator on
the right by the smaller value $\log(1/\epsilon_n)$ gives
$$\sup_{\epsilon\in[\epsilon_{n+1},\epsilon_n]}
\frac{\sup_{w\in U}|X(w,\epsilon)|}{\log(1/\epsilon)}
\leq
\sup_{\epsilon\in[\epsilon_{n+1},\epsilon_n]}\frac{
\sup_{w\in U}
\bigl|X(w,\epsilon)-X(w,\epsilon_n)\bigr|}{\log(1/\epsilon_n)}+\frac{\sup_{w\in U}|X(w,\epsilon_n)|}{\log(1/\epsilon_n)}.$$
Since the two terms on the right side tend to $0$ almost surely by the convergence statements above, 
$$
\sup_{\epsilon\in[\epsilon_{n+1},\epsilon_n]}
\frac{\sup_{w\in U}|X(w,\epsilon)|}{\log(1/\epsilon)}
\longrightarrow 0
\qquad\text{almost surely as }n\to\infty.
$$
Since the intervals $[\epsilon_{n+1},\epsilon_n]$ cover
$(0,1]$ as $n$ varies, this proves the proposition.
\end{proof}

\subsection{The Main Theorem}

We now prove Theorem~\ref{thm:main}.

\begin{proof}[Proof of Theorem~\ref{thm:main}]
By Proposition \ref{prop:uniformDiffGoal}, applied with
$U_n=B(0,n)$, we have that for every $n\geq 1$,
$$
\sup_{z\in U_n}\frac{|X(z,\epsilon)|}{\log(1/\epsilon)}
\to 0
\qquad\text{almost surely, as }\epsilon \to 0.
$$
Since there are only countably many $n$, these
convergences hold simultaneously almost surely. Now fix $z\in\mathbb R^d$ and choose $n$ such that
$z\in U_n$. Since
$$
\frac{h_\epsilon^\rho(z)}{\log(1/\epsilon)} =\frac{h_\epsilon^\sigma(z)}{\log(1/\epsilon)}
+\frac{X(z,\epsilon)}{\log(1/\epsilon)},
$$
and the final term tends to $0$ uniformly on $U_n$,
the $\liminf$, $\limsup$, and, when it exists, the limit of the two normalized fields coincide at $z$. Hence, almost surely,
$$
\mathcal T_{\alpha,\star}^\rho
=
\mathcal T_{\alpha,\star}^\sigma
\qquad
\text{for every }
\alpha\in\mathbb R
\text{ and every }
\star\in\{\liminf,\limsup,\lim\}.
$$
\end{proof}

\section{The 2D Zero-Boundary GFF on a Domain}\label{sec:zeroBoundary}

In this section we extend Theorem \ref{thm:main} to the zero-boundary Gaussian free field on an open domain with harmonically non-trivial boundary in two dimensions. The argument exploits the Markov property to decompose the normalized whole-plane GFF on the domain into a zero-boundary part (up to an additive constant) and an independent harmonic part.

Let $D\subset\mathbb R^2$ be an open domain with harmonically non-trivial boundary. Let $h$ denote the whole-plane GFF with additive constant fixed so that the average of $h$ over the unit circle is $0$. Let $h^D$ be a zero-boundary GFF on $D$. Let $\rho,\sigma$ be admissible mollifiers. For $\nu\in\{\rho,\sigma\}$ and $z\in D$, choose $0<\epsilon<\operatorname{dist}(z,\partial D)$. Since $\operatorname{supp}\nu\subseteq\overline{B(0,1)}$, the measure $\nu_{z,\epsilon}$ is supported inside $D$. We define
$$
h^{\nu}_\epsilon(z)=(h,\nu_{z,\epsilon}),\qquad
h^{D,\nu}_\epsilon(z)=(h^D,\nu_{z,\epsilon})
\qquad(\nu=\rho,\sigma).
$$
With the same abuse of notation as in Section \ref{sec:mollification}, we also write
$$
h^{\nu}_\epsilon(z)=\int h(z+\epsilon y)\,\nu(dy),\qquad
h^{D,\nu}_\epsilon(z)=\int h^D(z+\epsilon y)\,\nu(dy).
$$
Define the difference fields
$$
X(z,\epsilon):=h^{\rho}_\epsilon(z)-h^{\sigma}_\epsilon(z),
\qquad
X^{D}(z,\epsilon):=h^{D,\rho}_\epsilon(z)-h^{D,\sigma}_\epsilon(z).
$$
For $\nu\in\{\rho,\sigma\}$ and $\star\in\{\liminf,\limsup,\lim\}$, define
$\mathcal T_{\alpha,\star}^{D,\nu}$ by the corresponding condition in Definition \ref{def:thickPoints}, with $h^D$ in place of $h$. Applying the Markov property to $h$ (Lemma \ref{lem:markov}) with $U=D$, we obtain the decomposition
\begin{equation}\label{eq:markovDecomp}
h = \mathring h + \phi
\end{equation}
where $\phi$ is a random harmonic function on $D$, determined by the values of $h$ on $\mathbb C\setminus D$, $\mathring h$ is independent of $\phi$, and $\mathring h$ has the law of a zero‑boundary GFF on $D$ minus its average over $\partial\mathbb D\cap D$. 

Let $L$ be the average of $h^D$ over $\partial\mathbb D\cap D$, and $L = 0$ when $\partial\mathbb D\cap D = \emptyset$. Then, we have $\mathring h \stackrel{d}{=} h^D -L$. For any admissible mollifiers $\rho,\sigma$, we have, as processes in $(z,\epsilon)$,
\begin{equation}\label{hcircequivalence}
\mathring{h} * \rho_\epsilon(z) - \mathring{h}*\sigma_\epsilon(z) \overset{d}{=}\int (h^D -L)(z+\epsilon y)(\rho-\sigma)(dy) = X^{D}(z,\epsilon)
\end{equation}
Substituting the Markov decomposition into the difference field, we get
\begin{equation}\label{XinftyMarkov}
X(z,\epsilon) = (\mathring{h} * \rho_\epsilon(z) - \mathring{h}*\sigma_\epsilon(z)) + (\phi*\rho_\epsilon(z) - \phi*\sigma_\epsilon(z)).
\end{equation}
We now show that the harmonic correction term is negligible compared to $\log(1/\epsilon)$.

\begin{lem}\label{lem:harmonicSmall}
Let $K\subset D$ be a closed ball contained in $D$.
If $\rho,\sigma$ are admissible mollifiers, then
$$
\lim_{\epsilon\to0}\sup_{z\in K}\bigl|\phi*\rho_\epsilon(z)-\phi*\sigma_\epsilon(z)\bigr|=0.
$$
\end{lem}
\begin{proof}
Since $\phi$ is harmonic on $D$, it is smooth, hence uniformly continuous on a compact neighborhood of $K$ contained in $D$. Therefore,
$$
\sup_{z\in K}\sup_{x\in B(0,1)}|\phi(z+\epsilon x)-\phi(z)|\to 0\qquad\text{as }\epsilon\to0.
$$
Because $\rho$ and $\sigma$ are probability measures supported in $\overline{B(0,1)}$, it follows that both
$\phi*\rho_\epsilon(z)$ and $\phi*\sigma_\epsilon(z)$
converge uniformly on $K$ to $\phi(z)$ as $\epsilon\to0$.
Hence their difference converges to $0$ uniformly on $K$.
\end{proof}

We are now in a position to prove the equivalence of thick point sets for the zero-boundary GFF.

\begin{proof}[Proof of Theorem~\ref{thm:main-zero}]
By the same reduction as in the proof of Theorem \ref{thm:main}, it suffices to show that for every closed ball $K\subset D$, 
$$
\frac{\sup_{z\in K}\bigl|X^{D}(z,\epsilon)\bigr|}{\log(1/\epsilon)}\longrightarrow 0
\qquad\text{almost surely as }\epsilon\to0.
$$
This is enough, since $D$ can be covered by countably many such closed balls.

By \eqref{XinftyMarkov} we get,
$$|\mathring h*\rho_\epsilon(z) - \mathring h *\sigma_\epsilon(z)| = |X(z,\epsilon)-(\phi*\rho_\epsilon(z)-\phi * \sigma_\epsilon(z))| \leq |X(z,\epsilon)| + |\phi*\rho_\epsilon(z)-\phi * \sigma_\epsilon(z)|$$
Taking the supremum over $K$ and dividing by $\log (1/\epsilon)$, we get
$$\sup_{z \in K}\frac{|\mathring h*\rho_\epsilon(z) - \mathring h *\sigma_\epsilon(z)|}{\log (1/\epsilon)} \leq \sup_{z \in K}\frac{|X(z,\epsilon)|}{\log (1/\epsilon)} + \sup_{z \in K}\frac{|\phi*\rho_\epsilon(z)-\phi * \sigma_\epsilon(z)|}{\log (1/\epsilon)}$$
By assumption, both $\rho$ and $\sigma$ satisfy Condition \eqref{eq:secondMoment} with respect to the log-correlated Gaussian field $h$. The same argument as in Proposition \ref{prop:uniformDiffGoal}, applied on a bounded open neighborhood $U\subset D$ of $K$ to the LGF $h$ on $\mathbb{R}^d$, gives
$$
\frac{\sup_{z\in K}\bigl|X(z,\epsilon)\bigr|}{\log(1/\epsilon)}\longrightarrow 0
\qquad\text{almost surely as } \epsilon \to 0.
$$
By Lemma \ref{lem:harmonicSmall}, the term $\frac{\sup_{z\in K}\bigl|\phi*\rho_\epsilon(z)-\phi*\sigma_\epsilon(z)\bigr|}{\log(1/\epsilon)}$ tends to zero as $\epsilon \to 0$. We can conclude 
$$
\sup_{z \in K}\frac{|\mathring h*\rho_\epsilon(z) - \mathring h *\sigma_\epsilon(z)|}{\log (1/\epsilon)}\longrightarrow 0
\qquad\text{almost surely as }\epsilon\to0.
$$
Finally, note that
$$
\mathring h*\rho_\epsilon(z)-\mathring h*\sigma_\epsilon(z)\overset{d}{=}X^{D}(z,\epsilon)
$$
by \eqref{hcircequivalence}. Choose countably many closed balls $K_n\subset D$ with $D=\bigcup_{n\ge1}K_n$. The convergence proved above holds almost surely for every $K_n$ simultaneously. Hence, from the same reduction used at the beginning of the proof of Theorem \ref{thm:main}, almost surely
$$
\mathcal T_{\alpha,\star}^{D,\rho}
=
\mathcal T_{\alpha,\star}^{D,\sigma}
\qquad
\text{for every }
\alpha\in\mathbb R
\text{ and every }
\star\in\{\liminf,\limsup,\lim\}.
$$
\end{proof}

\section*{Acknowledgments} 
This paper was completed during the University of Chicago Math REU program. I would like to thank Professor Ewain Gwynne for his support and guidance throughout this program. He is a kind and inspiring mathematician, and it was a pleasure and privilege to work with him. This paper would not have been completed without his many helpful suggestions. I would like to thank Peter May for organizing the program, and Boston College for financially supporting my participation.

\end{document}